\documentclass[12pt]{article}
\usepackage{color}

\usepackage{amsfonts,amsmath,amsxtra,amsthm,amssymb,latexsym}
\usepackage{amssymb}  
\usepackage{bbm} 
\usepackage{graphicx}

 \usepackage[utf8]{inputenc}

\newcommand\blfootnote[1]{%
  \begingroup
  \renewcommand\thefootnote{}\footnote{#1}%
  \addtocounter{footnote}{-1}%
  \endgroup
}

\newtheorem{thm}{Theorem}[section]
 \newtheorem{cor}[thm]{Corollary}
 \newtheorem{lem}[thm]{Lemma}
 \newtheorem{prop}[thm]{Proposition}
 \theoremstyle{definition}
 \newtheorem{defn}{Definition}[section]
 \theoremstyle{remark}
 \newtheorem{rem}{Remark}[section]
  
 \newtheorem{ex}[thm]{Example}

 \numberwithin{equation}{section}

\let\emptyset\varnothing

\DeclareMathOperator{\im}{Im}
\DeclareMathOperator{\re}{Re}
\DeclareMathOperator{\dom}{dom}
\DeclareMathOperator{\cone}{Cone}

\DeclareMathOperator{\Bd}{bd}

\DeclareMathOperator{\diam}{diam}

\def\RR{\mathbb R}
\def\CC{\mathbb C}
\def\NN{\mathbb N}
\def\ZZ{\mathbb Z}

\def\DD{\mathbb D}

\def\SSS{\mathbb{S}}

\def\al{\alpha}

\def\la{\lambda}

\def\de{\delta}
\def\ep{\epsilon}

\def\ga{\gamma}
\def\Ga{\Gamma}
\def\pa{\partial}

\def\vka{\varkappa}
\def\Om{\Omega}

\def\De{\Delta}
\def\Si{\Sigma}

\def\L{\mathcal{L}}
\def\D{D}
\def\W{\mathcal{W}}
\def\B{\mathcal{B}}

\def\fr{f}
\def\dr{\mathfrak{r}}

\def\A{\mathbb{A}}

\def\F{\mathbb{F}}

\def\GaWidth{\varepsilon}

\def\ii{\mathrm{i}}

\def\ssdots{\scriptscriptstyle{\dots}}

\def\wt{\widetilde}

\def\res{\mathrm{res\;}}

\begin{document}
\title{Resonance free regions and non-Hermitian spectral optimization for Schrödinger point interactions}
\author{}
\date{}
\maketitle

\vspace{-8ex}

{\center 
{\large 
Sergio Albeverio and Illya M. Karabash 
\\[0mm]
}}

\blfootnote{\noindent\textbf{Acknowledgments.}
The authors are grateful to the anonymous referee for  careful reading of the paper and 
very helpful and stimulating remarks about exponential polynomials and quantum graphs. The authors 
are also  thankful to Gianfausto Dell'Antonio and Alessandro Michelangeli for the invitation to 
the 2nd workshop ''Mathematical Challenges of Zero-Range Physics'' (SISSA, Trieste), 
where some questions connected with this research were discussed with specialists. 
The second author (IK) is grateful to Richard Froese, who has introduced him into the area of resonance optimization, and to Herbert Koch, Daniel Peterseim, and Holger Rauhut for the hospitality of the University of Bonn.
Both authors are also thankful to Richard Froese for pointing out the recent thesis of Kai Ogasawara \cite{KO14_Th} 
written under his supervision. 
   
During various parts of this research, IK was supported by 
the Alexander von Humboldt Foundation, the Hausdorff Research Institute for Mathematics of the University of Bonn, and 
the  WTZ grant 100320049 ''Mathematical Models for Bio-Medical Problems'' jointly sponsored by BMBF (Germany) 
and MES (Ukraine). The participation of IK in the summer school ''Modeling, Analysis, and Approximation Theory toward applications in tomography and inverse problems'' (L{\"u}beck, July 31 - August 4, 2017) with a series of lectures based on these results    was supported by VolkswagenStiftung.
}

\begin{abstract}
Resonances of Schrödinger Hamiltonians with point interactions  
are considered. The main object under the study is the resonance free 
region  under the assumption that 
the centers, where the point interactions are located, are known and the associated 
`strength' parameters are unknown and allowed to bear additional dissipative effects. 
To this end we consider the boundary of the resonance free region 
as a Pareto optimal frontier and study 
the corresponding optimization problem for resonances. It is shown that upper logarithmic bound 
on resonances can be made uniform with respect to the strength parameters.
The necessary conditions on optimality are obtained in terms of first principal minors  
of the characteristic determinant. We demonstrate the applicability of these optimality 
conditions on the case of 4 equidistant centers  by computing explicitly 
the resonances of minimal decay for all frequencies.  
This example shows that a resonance of minimal decay is not necessarily simple, 
and in some cases it is generated by an infinite family of feasible resonators.
\end{abstract}

{
\small \noindent
MSC-classes: 
35J10, 
35B34, 
35P15, 
49R05, 
90C29, 
47B44, 
\\[2mm]

\tableofcontents
}

\section{Introduction}
\label{s:intro}

\subsection{Statement of problem, motivation, and related studies}
\label{subs:Statement}

In the present paper, we study resonance free regions and extremal resonances of 
`one particle, finitely many centers Hamiltonian' $H_\al=-\De_{\al,Y}$ 
associated with the formal expression 
$
-\De u (x) + \sum_{j=1}^N \mu (\al_j) \de (x - y_j) u (x) , \quad x \in \RR^3 , \ N \in \NN,
$, where $\De$ is the self-adjoint Laplacian acting in the complex Lebesgue space $\L^2 (\RR^3)$, $\delta (\cdot - y_j)$ is the Dirac measure at $y_j \in \RR^3$, $\mu(\al_j)$ is a complex-valued function of the strength parameter $\al_j$, $j=1,\dots,N$ 
(see \cite{AFH79,AGH82,AGHH12,AK00} and Section \ref{s:Def} for basic definitions).
The question of optimization of the principal eigenvalue of self-adjoint Schrödinger Hamiltonians with $\delta$-type or point interactions  attracted recently considerable attention especially in a quantum mechanics context \cite{E05,EHL06,EFH07,EL16,L16}. 
This line of research was motivated by the isoperimetric problem posed in \cite{E05}. 

In comparison with variational problems  involving eigenvalues of self-adjoint operators, the resonance spectral problem describes the dissipation of energy 
to the outer medium and so it is of a non-Hermitian type. 
The facts that resonances move under perturbations in two-dimensions of the complex plane 
and that degenerate (multiple) resonances can split in non-differentiable branches  lead to essentially new difficulties and effects  
for the application of variational techniques  \cite{HS86,S87,CZ95,CO96,KS08,Ka13,Ka14,KLV17}. 
In particular, the problem of optimization of an individual  resonance 
takes the flavor of Pareto optimization if one considers it as an $\RR^2$-valued objective  function and the boundary of the resonance free region as a Pareto frontier \cite{Ka14,KLV17}.
Numerical optimization of 1-D resonances produced by point interactions were initiated recently in \cite{KO14_Th}.

Estimates on poles of scattering matrices and resonances have being   
studied in Mathematical Physics at least since the Lax-Phillips 
upper logarithmic bound on resonances' imaginary parts \cite{LP71} and constitute an active area of research \cite{DZ16,F97,Z89}.
Optimization of resonances may be seen as an attempt to obtain 
sharp estimates on resonance free regions. 
This point of view and the study of resonances associated with 
random Schrödinger operators were initial sources of 
the interest in this problem \cite{H82,HS86,S87}. 

The present growth of interest in numerical 
\cite{GCh13,HBKW08,KS08,MPBKLR13,OW13} 
and analytical \cite{Ka13,Ka14,KLV17} 
aspects of resonance optimization 
is stimulated by a number of optical engineering studies of resonators with high quality factor (high-Q cavities), see  \cite{DMTSH14,LJ13,MPBKLR13,NKT08} and references therein.

In this paper, we assume that the tuple of centers $Y = (y_j)_1^N \in (\RR^3)^N$ (locations of the $\delta$-interactions) is fixed and known, but the N-tuple $\al= (\al_j)_1^N$ of scalar free 'strength' parameters $\al_j$ of point interactions is unknown. 
The associated point interactions Hamiltonians $H_\al=-\De_{\al,Y}$ can be defined in several ways as densely defined closed operators in the Hilbert space $\L^2 (\RR^3)$ \cite{AGHH12,AH84,G84}, in particular, via a Krein-type  formula for 
the difference 
of the perturbed and unperturbed resolvents of operators $H_\al$ and $-\De$, respectively. 
Eigenvalues and (continuation) resonances $k$ of the corresponding operator $H_\al$ are connected with the special $N\times N$-matrix function $\Ga_{\al,Y} (z)  $ which appears naturally as a part of the expression for $(-\De_{\al,Y} - z^2 )^{-1} - (-\De - z^2 )^{-1} $,
see Section \ref{s:Def}.
If one denotes by $\Si (\al,Y)$ the set of zeroes of 
$
\det \Ga_{\al,Y} (\cdot) 
$, 
then the set $\Si_\res (\al,Y)$ of resonances $k$ associated with $H_\al$ 
can be defined by 
\begin{equation} \label{e:SiRes}
\Si_\res (\al,Y) := \Si (\al,Y) \cap (\CC_- \cup \RR)  , \quad \text{see \cite{AGHH12,AH84}},
\end{equation}
where $\CC_-$ is the lower half of the complex plane.

The functions $ \det \Ga_{\al,Y} (\cdot) $ take the form of exponential polynomials, for those there exists a 
well-developed theory with a number of applications in Analysis and 
connections to the studies of the Riemann zeta function \cite{BG12,L31,MSV13}.  
P\'{o}lya's results on positions and distribution of zeros of exponential polynomials 
were refined and generalized in many works leading, in particular, to the P\'{o}lya-Dickson theorem \cite{BG12}. 
This theorem implies, for example, that the imaginary parts of resonances of $H_\al$ satisfy upper 
and lower logarithmic  bounds (see Lemma \ref{l:asDetGa} and  (\ref{e:beLogDet}) below), 
in this way establishing and strengthening for point interactions the Lax-Phillips result \cite{LP71}. 
From this point of view, the present work can be seen as an attempt to obtain  more refined bounds 
on zeros of special exponential polynomials employing Pareto optimization techniques of \cite{Ka13,Ka14,KLV17}.

While our main  goal is to consider the resonance free regions in the case where the $\al_j $ run through the compactification
$\overline{\RR} := \RR \cup \{ \infty \}$ of the real line, our technique also leads us to the study of `dissipative point interactions' corresponding to the case $\al_j \in \overline{\CC}_-:= \CC_- \cup \RR \cup \{\infty\}$. 
It is not difficult to see  (see Section \ref{s:Def}) that the corresponding  operators $H_\al$ are well-defined, closed, and maximal dissipative 
in the sense that the $\ii H_\al$ are maximal accretive (i.e., $\re (\ii H_\al u,u) \ge 0$ 
for all $u$ in the domain $\dom H_\al$ of $H_\al$ and $(\ii H_\al + \la) \dom H_\al = \L^2 (\RR^3) $ for $\la >0$). 
So $H_\al$ can be considered as  pseudo-Hamiltonians in the terminology of \cite{E12}. 
Following the logic of the resolvent continuation 
it is natural to extend the definition of resonances given by  formula (\ref{e:SiRes}) to the case $\al \in (\overline{\CC}_-)^N $.

Assuming that each of the parameters $\al_j$, $j=1, \dots, N$, is allowed to run 
through some  set $\A \subset \overline{\CC}_-$ we consider the associated operators $H_\al$ as feasible points (see \cite{BV04} for basic notions of the optimization of vector-valued objective functions) and denote the associated feasible set of operators  by 
$\F_\A$. The resonance free region for the family $\F_\A$ is defined as  $\CC \setminus \Si_\res [\F_\A]$
where $\Si_\res [\F_\A] := \bigcup_{H_\al \in \F_\A} \Si_\res (\al,Y) $ is the \emph{set of achievable resonances}.

\subsection{Main results and some examples}

The main results of the present paper are:
\begin{itemize}
\item  It is shown in Theorem \ref{t:UniLogBound} that upper logarithmic bounds on imaginary parts of 
resonances can be modified to become uniform estimates over $\F_{\overline{\CC}_-}$ and $\F_{\overline{\RR}}$.

\item To achieve more detailed results on the resonance free region, we employ the Pareto optimization approach  and consider 
Hamiltonians $H_\al \in \F_\A $ that produce resonances on the boundary $\Bd \Si_\res [\F_\A]$ of the set of achievable resonances. When the set $\A$ of feasible strength parameters $\al_j$, $j=1$, \dots, $N$, is closed in the topology of the compactification $\overline{\CC}_-$, such extremal feasible operators $H_\al$ do exist
since the set $\Si_\res [\F_\A]$ is closed (see Theorem \ref{t:SAclosed}).
The function of minimal decay rate $\dr(\cdot)$ \cite{Ka13} provides a convenient way to describe  the part of $\Bd \Si_\res [\F_\A]$ closest to $\RR$ (see Definition \ref{d:MinDec} and the discussions in Section \ref{s:AddRem}).  The associated extremal resonances $k$ and operators $H_\al$ are said to be of minimal decay for their particular frequencies $\fr = \re k$.

\item In Section 6 we obtain various necessary conditions on $H_\al$ to be extremal over $\F_{\overline{\RR}}$ and $\F_{\overline{\CC}}$ in terms of first minors of a regularized version of $\det \Ga_{\al,Y}$. This is done with the use of the multi-parameter perturbations technique of \cite{Ka14}.

\item The effectiveness of the conditions of Section 6  can be seen in  the equidistant cases when $|y_j - y_{j'}|=L$ for all $j \neq j'$. Namely, we provide an explicit calculation of resonances of minimal decay and associated tuples $\al$ for the case where
$\{ y_j\}_1^4$ constitute the vertices of a regular tetrahedron (see Section \ref{s:Tetrahedron}). 
\end{itemize}

In the process of deriving the above results, we obtained several examples that are of independent interest since they address the questions arising often in the study of resonances and their optimization.

Namely, it occurs in the case of vertices of a regular tetrahedron that the optimal $\al$ does not always consist of equal $\al_j$ and that, for some of resonances of minimal decay, there exists an infinite family of optimizers $H_\al$ preserving only one of the symmetries (see the discussion in Section \ref{s:AddRem}). This gives a negative answer to the multidimensional part of 
the  question of uniqueness of optimizers for a given $\re k$, which was posed in \cite[Section 8]{Ka14} (see also \cite{HS86,KLV17_Im}).

The assumption that a resonance $k$ is of multiplicity 1 essentially simplifies its perturbation theory (see (\ref{e:k_asy})), 
and therefore this assumption is often explicitly or implicitly used in intuitive arguments. 
While it is known that generic  resonances are simple \cite{DZ16} (i.e., of multiplicity 1), 
there are no reasons to assume that resonances of minimal decay are generic. 
Example \ref{ex:Tetrahedron} describes $ H_\al \in \F_{\overline{\RR}}$ that produce  resonances of minimal decay 
with  multiplicity $\ge 2$. 

Nonzero resonances on the real line  are often assumed to be connected with eigenvalues embedded into the essential spectrum. 
Remark  \ref{r:RealRes} provides a very simple example of a dissipative Schrödinger Hamiltonian that generates 
a resonance $k$ in $\RR_-$, but has no embedded eigenvalue at $k^2$.

\textbf{Notation}. 
The following standard sets are used: the lower ($-$) and upper ($+$) complex half-planes 
$\CC_\pm = \{ z : \pm \im z > 0  \}$,
$\CC_{\mathrm{I}}$, $\CC_{\mathrm{II}}$, $\CC_{\mathrm{III}}$, and $\CC_{\mathrm{IV}}$ are the open quadrants in $\CC$
 corresponding  to the combinations of signs $(+,+)$, $(-,+)$, $(-,-)$, and $(+,-)$ for $(\re z,\im z)$, 
open half-lines $\RR_\pm = \{ x \in \RR: \pm x >0 \}$, 
open discs $\DD_\epsilon (\zeta) := \{z \in \CC : | z - \zeta | < \epsilon \}$,
and the boundary $\Bd S$ of a subset $S$ of a normed space $U$.
For $u_0 \in U$ and $z \in \CC$, we write 
$
z S  + u_0 := \{ zu + u_0 \, : \, u \in S \}
$. 
The convex cone generated by $S$ (all nonnegative linear combinations of elements of $S$) is denoted by $\cone S$.
If a certain map $g$ is defined on $S$, $g[S]$ is its image (when it is convenient, we write 
without brackets, e.g. $\re S$ for $S \subset \CC$.) 
The diameter of $S$ is $\diam (S) := \sup_{u_{0,1} \in S } \|u_0 - u_1\|_U$.
By  $\pa_z f $, $\pa_{\al_j} f$, etc., we denote (ordinary or partial)
derivatives  with respect to (w.r.t.) $z$, $\al_j$, etc.;
$\deg p$ stands for the degree of a polynomial $p$ of one or several variables.

\section{Nonconservative point interactions}
\label{s:Def}

Let us fix a set $Y=\{ y_j \}_{j =1}^N$ consisting of $N$ distinct points 
$y_1$, \dots, $y_N$ in $\RR^3$.
For every tuple $\al = (\al_j)_{j=1}^N \in \RR^N$, there exists 
the self-adjoint Hamiltonian $H_\al= -\De_{\al,Y}$ in $\L^2 (\RR^3)$ with point interactions at the centers $y_j$ that has  for all 
$z \in \CC_I$
the resolvent $(-\De_{\al,Y}-z^2)^{-1}$ with the integral kernel 
\begin{gather} \label{e:Res}
(-\De_{\al,Y}-z^2)^{-1} (x,x')  = G_z (x-x') + \sum_{j,j' = 1}^N G_z (x-y_j)
\left[ \Ga_{\al,Y} \right]_{j,j'}^{-1} G_z (x' - y_{j'} ) , 
\end{gather}
where $x,x' \in \RR^3 \setminus Y$ and $x \neq x'  $, see \cite[Section II.1.1]{AGHH12}.
Here $G_z (x-x') := \frac{e^{\ii z |x-x'|}}{4 \pi |x-x'|}$ is the integral kernel associated the resolvent $(-\De - z^2)^{-1}$ of 
the kinetic energy Hamiltonian $-\De$,
and $\left[ \Ga_{\al,Y} \right]_{j,j'}^{-1}$ denotes the $j,j'$-element of the inverse to 
the matrix 
\begin{gather*} 
\Ga_{\al,Y} (z) = \left[ \left( \al_j - \tfrac{\ii z}{4 \pi} \right) \de_{jj'} 
- \wt G_z  (y_j-y_{j'})\right]_{j,j'=1}^{N} \text{ with }
\wt G_z (x) := \left\{ \begin{array}{rr} G_z (x), & 
x \neq 0 \\
0 , & 
x = 0  \end{array} \right.  .
\end{gather*}

In the case of one center ($N=1$) and $\al_1 \in \CC$, the above definition 
leads to the m-accretive operator $\ii H_{\al_1}$ when $\al_1 \in \CC_-$, and the m-accretive operator $(-\ii) H_{\al_1}$ when $\al_1 \in \CC_+$ (see \cite{AGHS83} and \cite[Sections I.1.1 and I.2.1]{AGHH12}).



The aim of this section is to extend  the above definition to all tuples $\alpha \in \CC^N $.
Later we will use the case $\al  \in (\CC_-  \cup \RR)^N $ as a technical tool for optimization of resonances over $\al \in \RR^N$. 

Here and below $\deg p$ is the degree of the polynomial $p$ of one or several variables and $\diam (Y) := \max_{1 \le j,j'\le N} |y_j - y_{j'}|$ is the diameter of $Y$.

As it was pointed out to us by the referee, the following lemma could be obtained from the theory of  zeroes of exponential polynomials \cite{L31,BG12} 
which goes back to P\'{o}lya.
We provide here a short self-contained proof that while not using the general theory, shows how one of P\'{o}lya's arguments works.

\begin{lem} \label{l:asDetGa}
For every $\al \in \CC^N$, there exist $c_{i,j}= c_{i,j} (\al,Y)>0$, $i,j=1,2$,   such that
all zeros $k$ of $\det \Ga_{\al,Y} (\cdot)$ satisfy 
\begin{equation}
- c_{2,1} \ln (|\re k|+1) - c_{2,2} \le \im k \le - c_{1,1} \ln (|\re k|+1) + c_{1,2} . \label{e:beLog}
\end{equation}
\end{lem}

\begin{proof}
Consider $\det \Ga_{\al,Y} (z)$ as a function in $z$ only. Then 
there exists a unique representation 
\begin{equation} \label{e:det=Sexp}
\det \Ga_{\al,Y} (z) = (-4\pi)^{-N} D (z) , \quad D(z) = \sum_{l=0}^\nu p_l (z) e^{\ii z q_l} ,
\end{equation}
where  the numbers $\nu = \nu (\al,Y) \in \NN \cup \{0\}$, $q_l = q_l (\al,Y) \ge 0$, and the nontrivial polynomials $p_l (z)$  (i.e., $p_l \not \equiv 0$) with coefficient depending on $\al$ and $Y$ are such that
\begin{gather*}
0 = q_0 < q_1 < \dots < q_\nu \le N \diam (Y).
\end{gather*}
Clearly, $p_0 (z) =  \prod_{j=1}^N  (\ii z - 4 \pi \al_j )$ and $\deg p_l \le N-2  \text{ for all } 1 \le l \le \nu.$

 If $\nu=0$, $\det \Ga_{\al,Y} (z) =  p_0 (z)$ and so the statement of the lemma is obvious.
 
Note that 
$\nu =0$  if and only if  $N=1$.
Indeed, for $N \ge 2$ it is easy to see that $q_1 = 2 \min_{j \neq j'} |y_j - y_{j'}|$ and the terms containing $e^{\ii z q_1}$ do not cancel.

Let $N\ge 2$ and $\nu \ge 1$. We prove (\ref{e:beLog}) in several regions of $\CC$ and then take the largest of the corresponding constants $c_{i,j}$.
First, note that (\ref{e:beLog}) is obvious in any disc $\DD_r (0)$ and also for $z \in \CC_+ \cup \RR$ (due to asymptotics of exponential terms 
in (\ref{e:det=Sexp})).

Let $z \in \CC_-$. Then 
there exists $ r_1 (\al)>0$ and $ C_1 (\al,Y)>0$ so that
\[
 |D(z) | \ge |p_0 (z)| - C_1  |z|^{N-2} |e^{\ii z q_\nu}| \ge 2^{-1} |z|^{N-2} 
\left[ (|z|+1)^2 -2 C_1 e^{-q_\nu \im z}\right] \quad \text{for $|z|\ge r_1 $.}
\]
Assuming additionally $z \in \Om_1 =\{ \im z > - c_{1,1} \ln (|\re z|+1) + c_{1,2} \}$, we  see that 
\[
 (|z|+1)^2 -2 C_1 e^{-q_\nu \im z} \ge (|z|+1)^2 - (|\re z|+1)^{q_\nu c_{1,1} } 2C_1 e^{-q_\nu c_{1,2}} >0 
\]
whenever $C_1 e^{-q_\nu c_{1,2}} \le 1/4$ and $c_{1,1} \le 2/q_\nu$.
Hence, such a choice of $c_{1,1}, c_{1,2}$ ensures the absence of zeros of $D$ in $(\Om_1 \cap \CC_-) \setminus \DD_{r_1} (0)$.

On the other hand, for certain $r_2 (\al)>0$, $C_3 (\al,Y)>0$, and  $C_4 (\al,Y)>0$, it follows from $|z|\ge r_2$ that 
\begin{equation*}
|D(z)| \ge C_3 e^{-q_\nu \im z} - C_4  |z|^N e^{- q_{\nu-1} \im z} = 
C_3 e^{- q_{\nu-1} \im z} (e^{(q_{\nu-1}-q_\nu) \im z } - |z|^N C_4/C_3).
\end{equation*}
Thus, taking $c_{2,1} \ge N/(q_\nu-q_{\nu-1})$ it is easy to show the existence of  $c_{2,2}$ and $ r_3 (\al,Y) > r_2$ such that $z \in \Om_2 =\{ \im z < - c_{2,1} \ln (|\re z|+1) - c_{2,2} \}$ and $z \in \CC_- \setminus \DD_{r_3} (0)$ imply $|D(z)|>0$.
\end{proof}

\begin{prop} \label{p:Res}
Let $\al  \in \CC^N$. Then there exists a closed operator $H_\al = -\De_{\al,Y}$ in $\L^2 (\RR^3)$ with the spectrum 
$\sigma (H_\al) = [0,+\infty) \cup \{z^2 : z \in \CC_+ , \ \det \Ga_{\al,Y} (z) =0\} $
and the resolvent $(H_\al - z^2)^{-1} $ defined for $\{ z \in \CC_+ : z^2 \not \in \sigma (H_\al) \}$ by the integral kernel 
(\ref{e:Res}). If $\al \in (\CC_-  \cup \RR)^N $, the operator $\ii H_\alpha$ is 
m-accretive in the sense of \cite{Kato13}.
\end{prop}

\begin{proof}
The proof of the first statement can be obtained by modification of the arguments of \cite[Section II.1.1]{AGHH12}
with the use of Lemma \ref{l:asDetGa} and the formula 
\begin{gather}
\text{$(\Ga_{\al,Y} (z))^* = \Ga_{\overline{\al},Y} (-\overline{z})$, where $\overline{\al} := (\overline{\al}_j)_{j=1}^\infty$} \label{e:Ga*}
\end{gather}
(here $\overline{z}$ is the complex conjugate of $z \in \CC$).

Let now $\al \in (\CC_-  \cup \RR)^N $. Then, it is easy to see that, for 
$z \in \ii \RR_+$ the operator $\ii \Ga_{\al,Y} (z) $ is  accretive in the $N$-dimensional $\ell^2$-space. So, if additionally $\det \Ga_{\al,Y} (z) \neq 0$,  the operator
$(\ii H_\al - \ii z^2)^{-1}$ and, in turn, $\ii H_\al $ are accretive. 
Since the resolvent set of $H_\al$ is nonempty, $\ii H_\al $ is m-accretive.
\end{proof}

\section{Resonances and related optimization problems}
\label{s:ResAndOpt}

We will use the compactifications $\overline{\CC} =  \{ \infty \} \cup \CC$, $\overline{\RR}= \{ \infty \} \cup \RR$, and $\overline{\CC}_- := \{ \infty \} \cup \RR \cup \CC_- $.

To carry over the above definitions of point interactions to 
the extended $N$-tuples 
$\al \in \overline{\CC}^N$,
we put, following \cite{AGHH12}, $\De_{\al,Y}=\De_{\wt \al, \wt Y}$, 
where 
\begin{multline} \label{e:wtAl}
\text{$\wt \al$ and $\wt Y$ are produced from $\al$ and $Y$, resp.,} \\ 
\text{by removing of the components with numbers $j$ satisfying $\wt \al_j = \infty$.} 
\end{multline}
(It is assumed in the sequel that if all $\al_j \in \CC$, then $\wt \al := \al$, $\wt Y:=Y$). 
Using this rule we can formally define the function $\det \Ga_{\al,Y} (\cdot) := \det \Ga_{\wt \al, \wt Y} (\cdot)$
for arbitrary $\al \in \overline{\CC}^N$.

Points $k$ belonging to the set  $\Si (\al,Y)$ of zeroes of the 
determinant 
$
\det \Ga_{\al,Y} (z) 
$
will be called $\Ga^{-1}$-poles (or $\Ga_{\al,Y}^{-1}$-poles). 
The set of (continuation) resonances $\Si_\res (\al,Y)$ associated with $H_\al$ 
is defined by (\ref{e:SiRes}).
(This definition is in agreement with the case of real $\al_j$ considered in \cite{AGHH12,AH84,G84}, where also the connection of $\Ga_{\al,Y}^{-1}$-poles in $\CC_+$ with eigenvalues of $H_\al$ is addressed.
For the origin of this and related approaches to the understanding of resonances, we refer
to \cite{AH84,DZ16,G84,RSIV78,V72} and the literature therein). 

The multiplicity of a resonance or a $\Ga^{-1}$-pole will be understood as the multiplicity 
of a corresponding zero of the analytic function $\det \Ga_{\al,Y} (\cdot)$ (see \cite{AGHH12}).

For fixed $Y$, consider the set 
\begin{equation} \label{e:F}
\F = \{ -\De_{\al,Y} \ : \ \al \in S \}
\end{equation}
 of operators $H_\al$ with  $N$-tuples $\alpha$ belonging 
to a certain set $S \subset \overline{\CC}^N $.
Let us introduce the sets of all possible 
resonances $\Si_\res [\F]$ and $\Ga^{-1}$-poles $\Si [\F]$ generated by $H_\al \in \F$,
\begin{gather*} 
\Si [\F] := \bigcup_{-\De_{\al,Y} \in \F} \Si (\al,Y) , \quad
\Si_\res [\F] := \bigcup_{-\De_{\al,Y} \in \F} \Si_\res (\al,Y) .
\end{gather*}

We consider $\F$ as \emph{a feasible set} \cite{BV04} of operators.
The main attention will be paid to the direct products $S=\A^N$ of  \emph{the sets $\A \subset  \overline{\CC}_-$ of feasible dissipative $\al_j$-parameters}. For these direct products, we employ the  notation 
$
\F_\A :=  \{ H_\al \ : \ \al \in \A^N  \} .
$ 

Our main goal is to find  resonances $k$ which are extremal over $\F_{\overline{\RR}}$ or 
$\F_{\overline{\CC}_-}$ in the framework of the  Pareto optimization approach of
\cite{Ka13,Ka14,KLV17}. 
In a wide sense,  resonances globally Pareto extremal over $\F$ can be understood 
as boundary points of \emph{the set of achievable resonances} $\Si_\res [\F]$.
Depending on the applied background of more narrow optimization problems,
various parts of the boundary $\Bd \Si_\res [\F]$ can be perceived as optimal resonances (see the discussion in Section \ref{s:AddRem} and in \cite[Section A.2]{KLV17}). Note that our definitions are slightly different from those in \cite{BV04}. In particular, from our point of view, the use of positive cones for the definition of Pareto optimizers is sometimes too restrictive for the needs of resonance optimization.

One of particular optimization
problems can be stated in the following way.
If $k \in \Si_\res (\al,Y)$ is interpreted as a resonance of the wave-type equation $\pa_t^2 u - \De_{\al,Y} u = 0$ (cf. \cite{LP71}) with `singular potential term $V=\sum_{j=1}^N \mu (\al_j) \de (x - y_j) u (x)$ ', then $\fr = \re k$ can be understood  
as a (real) frequency  of the associated resonant mode and $\dr = - \im k \ge 0$ is the 
corresponding exponential rate of decay (cf. \cite{DZ16,Ka13,Ka14}).

We say that $\fr \in \RR$ is an achievable frequency if $\alpha  \in \re \Si_\res [\F]$.
The properties of the set $\re \Si_\res [\F_{\overline{\RR}}]$ are discussed in 
Section \ref{s:AddRem}.

\begin{defn}[see \cite{Ka13} for 1-D resonances] \label{d:MinDec}
Let $\fr \in \re \Si_\res [\F]$. 
\emph{The minimal decay rate $\dr_{\min} (\fr) = \dr_{\min} (\fr;\F)$ for the frequency} $\fr$
is defined by
\[
 \ 
\dr_{\min} (\fr;\F) := \ \inf \{ \dr \in [0,+\infty) \ : \ \fr - \ii \dr \in \Si_\res [\F] \} . \ 
\]
If $k = \fr - \ii \dr_{\min} (\fr)$ is a resonance of a certain feasible operator $H_\al \in \F$
(i.e., the minimum is achieved), 
we say that $k$, $H_\al$, and $\al$ are of \emph{minimal decay for} $\fr$. 
\end{defn}

\begin{ex} \label{ex:1center}
Let $N=1$ and $Y=\{y_1\}$. Then  $\Si (\al_1,Y)$ consists of one $\Ga^{-1}$-pole $k=-\ii 4 \pi \al_1$ of multiplicity 1 \cite{AGHS83,AGHH12}.

\item[(i)] In the case $\A= \RR$, one has $\Si [\F_{\RR}] = \ii \RR := \{ i t  \ : \ t \in \RR \}$ and 
$\Si_\res [\F_{\RR}] = \ii (-\infty,0] $. The function $\dr_{\min} (\cdot; \RR)$ is defined 
on the set of achievable  frequencies consisting of one point $\re \Si_\res [\F_\RR] = \{ 0 \}$
and one has $\dr_{\min} (0; \F_\RR) = 0$. The resonance $k=0$ and the operator $H_0$ are of minimal decay
for the frequency $0$.

\item[(ii)] Let $\A  = \overline{\CC}_-$. Then $\Si [\F_{\overline{\CC}_-}] = \ii \CC_+ \cup \ii \RR$.
For each $\fr \in (-\infty,0]$, we have $\dr_{\min} (\fr; \F_{\overline{\CC}_-}) = 0$, and see that $k=\fr$ and
$H_{\ii \fr (4 \pi)^{-1}}$ are the resonance and an operator of minimal decay for $\fr$.
\end{ex}

\begin{rem} \label{r:RealRes}
It follows from Example \ref{ex:1center} (ii)  that,  in the dissipative case, nonzero real  resonances are not necessarily associated with embedded eigenvalues of $H_\al$.
Indeed, taking $N=1$ and $\al_1 \in \ii \RR_-$, we see that there exists a real resonance $k_0 <0$. The fact that $k_0^2$ is not an eigenvalue of $H_\al$ follows easily from the proof of \cite[Theorem I.1.1.4]{AGHH12}.
\end{rem}

\section{Existence of optimizers and perturbation theory}
\label{s:Da}

For every $a = (a_1, \dots, a_N) \in \overline{\CC}^N $, let us denote 
\begin{gather}
\text{by $n (a)$ the number of parameters $a_j$, $j=1, \dots,N$, that
are not equal to $\infty$,} \label{e:n}
\\ 
\text{and, for $k \in \Si [\F_\A]$, by }
n_{\min } (k; \F_\A) := \min \{ n (a) \ : \ k \in \Si (a,Y) \text{ and } a \in \A^N \}  \label{e:nmin}
\end{gather}
the minimal number of centers needed to generate $k$ over $\F_\A$.

Let us introduce on the compactification $\overline{\CC}$ of $\CC$ a metric $\rho_{\overline{\CC}} (z_1,z_2)$ generated by 
the stereographic projection and, 
e.g., the $\ell^2$-distance on the unit sphere $\SSS_2 \subset \RR^3$.
The direct product $\overline{\CC}^N $ will be considered as a compact metric space with the distance 
$\rho_{\overline{\CC}^N} (\alpha, \alpha')$
generated by the $\ell^2$-distance on $\SSS_2^N \subset \RR^{3N}$. 

Recall that, for $ S \subset \overline{\CC}^N$, the feasible set $\F$ of  operators is defined by (\ref{e:F}), and that $\Si [\F] =  \bigcup_{\al \in S } \Si (\al,Y)$ is the corresponding set of achievable  $\Ga^{-1}$-poles.

\begin{thm} \label{t:SAclosed}
Let the set $ S $ be closed in the metric space $\left( \overline{\CC}^N, \rho_{\overline{\CC}^N} \right)$. Then $\Si [\F] $ is a closed set and, for every achievable  frequency, there exists an operator 
$H_\al \in \F$ of minimal decay (in the sense of Definition \ref{d:MinDec}).
\end{thm}

This theorem easily follows from the 
following lemma and the compactness argument.
Recall that $\wt \al$ and $\wt Y$ are defined 
by (\ref{e:wtAl}).

\begin{lem} \label{l:LocDet}
For every $a \in \overline{\CC}^N $ there exists 
an open neighborhood 
$\W \subset \overline{\CC}^N$ of $a$ (in the topology of $(\overline{\CC}^N,\rho_{\overline{\CC}^N})$), an open set $\B \subset \CC^N$, 
a homeomorphism $\beta : \W \to \B$, and an analytic 
function $\D_{a} : \B \times \CC \to \CC $ such that, for every 
$\alpha \in \W$, the sets of zeroes of the function 
$\det \Ga_{\wt \al, \wt Y} (\cdot): \CC \to \CC$ coincide with  
 the sets of zeroes of the function $\D_{a} (\beta (\al) ; \cdot) : \CC \to \CC$ 
taking multiplicities into account.
\end{lem}

\begin{proof}
When $n(a) = N$ (see (\ref{e:n})), the lemma is obvious with $\beta (\al) \equiv \al$
and \linebreak $\D_{a} (\beta (\al) ; z) \equiv  \det \Ga_{\alpha, Y} (z)$.
Now, let us prove the lemma for the case $n := n (a) < N$. 
Without loss of generality, we can assume 
that $ a_j \in \CC$ for $1 \le j \le n (a)$ and $a_j = \infty$ for $j >  n (a)$.
Put $\beta_j = \alpha_j$ for $1 \le j \le n (a)$. For $n (a) +1 \le j \le N$
and $\al_j \neq 0$, let us define $\beta_j = -1/\alpha_j$ (assuming $1/\infty = 0$).
Then the  following regularized determinant 
\begin{equation} \label{e:Da} \text{\scriptsize
$ \D_{a}  = 
(-1)^n  \left| 
\begin{array}{cccccc} 
\frac{\ii z}{4 \pi}-\beta_1  & \ssdots &  G_z  (y_1-y_n) &  
 G_z  (y_1-y_{n+1}) & \ssdots &   G_z  (y_1-y_N) 
\\
\ssdots & \ssdots & \ssdots  &   \ssdots & \ssdots & \ssdots 
\\
  G_z  (y_n- y_1) & \ssdots & \frac{\ii z}{4 \pi} - \beta_n  & 
  G_z  (y_n - y_{n+1}) & \ssdots &  G_z  (y_n -y_N) 
 \\
  \beta_{n+1} G_z  (y_{n+1}- y_1) & \ssdots &  \beta_{n+1} G_z  (y_{n+1}- y_{N})  & 
\frac{\ii z \beta_{n+1}}{4 \pi}+1  & \ssdots & \beta_{n+1} G_z  (y_{n+1} -y_N) 
\\
\ssdots & \ssdots & \ssdots  &   \ssdots & \ssdots & \ssdots 
\\
  \beta_N G_z  (y_N- y_1) & \ssdots &   \beta_N G_z  (y_N- y_{n})  & 
\beta_N G_z  (y_N -y_{n+1}) & \ssdots &  \frac{\ii z \beta_N}{4 \pi}+1 
\end{array}  \right| $
}
\end{equation}
satisfies the conditions of the lemma.
\end{proof}

Lemma \ref{l:LocDet} allows one to consider the corrections of 
resonances and eigenvalues of $H_\al $ under small perturbations 
of $a \in \overline{\CC}^N$.
The first correction terms under one-parameter perturbations can be described in the following way.

Let $k$ be an $m$-fold zero of 
 the determinant $D_a (b; \cdot)$ defined 
by (\ref{e:Da}) at \linebreak $b=(a_1,\dots,a_n, 0, \dots, 0)$
and considered as an analytic 
function of the variables $z \in \CC$ and $\beta \in \CC^N$.
Then, for every analytic function $\ga (\zeta)$ that maps $\DD_r (0) \subset \CC$ to $\CC^N $ and satisfy 
$\ga (0) = b$, there exist $\ep >0 $, $\de>0$,
and continuous on $[0,\ep)$ functions $\kappa_j (\zeta)$, $j=1,\dots,m$, 
with the asymptotics 
\begin{equation} \label{e:k_asy}
\kappa_j (\zeta) = k + (C_{\ga,1}  \zeta)^{1/m} + o (\zeta^{1/m}) \  \text{ as $\zeta \to 0$}, 
\quad   
C_{\ga,1}  :=  -\frac{m! \, \pa_\zeta D_a (\ga(\zeta),k) \left. \right|_{\zeta = 0}}{\pa_z^m 
D_a (b,k)}  ,
\end{equation}
such that  all the zeros of 
$D_a (\ga (\zeta), \cdot)$, $ \zeta \in [0,\ep) $, 
lying  in $\DD_{\de} (k)$ 
are given by $\{ \kappa_j (\zeta) \}_1^m$
taking multiplicities into account.
In the case $C_{\ga,1} \neq 0$, 
each branch of $ [\cdot]^{1/m}$ corresponds
to exactly one of functions $\kappa_j$, and so, all $m$ values of functions  $\kappa_j (\zeta)$ for small enough $\zeta>0$
are distinct zeros of $D_a (\ga(\zeta), \cdot)$ of multiplicity 1.

Perturbations of $b$ in the directions of modified parameters $\beta_j$ play a special role.
Note that $D_a (\beta_1, \dots \beta_n, b_{n+1}, \dots, b_N;z) = 
D_a (\al_1, \dots \al_n, 0 , \dots, 0 ;z) = 
\det \Ga_{\wt \al, \wt Y} (z) $. 
Let $k \in \Si (a,Y)$.
If $a_i \in \CC$ (and so $i  \le n (a)$ and $\beta_i = \al_i$, under the convention of Lemma \ref{l:LocDet}), then 
the term $\pa_\zeta D_a (\ga(0),k)$ corresponding to the perturbation of one of the $\beta_j$ takes the form of the first principal minor
\begin{gather} \label{e:pa_alDa}
\pa_{\beta_i} \D_{a} (b;k) = \pa_{\al_i} \det \Ga_{\wt a, \wt Y} (k) = 
\det \Ga^{[i]}_{\wt \al, \wt Y} (k), \\
\text{ where } 
 \Ga^{[i]}_{\wt \al, \wt Y} (k) := \left[ \left( \al_j - \frac{\ii k}{4 \pi} \right) \de_{jj'} 
- \wt G_z  (y_j-y_{j'})\right]_{\substack{j,j'=1, \dots , n \\ 
j,j' \neq i}} . \notag
\end{gather}
If $a_i = \infty$ (and so $i > n$, $b_i = 0$, and $\beta_i = -1/\al_i$),
one has
\begin{equation} \label{e:pa_betaNDa}
 \pa_{\beta_i} D_a (b,k)  = 
(-1)^{n}  \left| 
\begin{array}{cccc}
\frac{\ii k}{4 \pi}-\beta_1  & \ssdots &  G_k  (y_1-y_{n}) &  
 G_k  (y_1-y_i) 
\\
\ssdots & \ssdots & \ssdots  &   \ssdots 
\\
  G_k  (y_{n} - y_1) & \ssdots & \frac{\ii k}{4 \pi} - \beta_{n}  & 
  G_k  (y_{n} - y_i ) 
 \\
  G_k  (y_i- y_1) & \ssdots &   G_k  (y_i- y_{n})  & 
  c
\end{array}  \right| 
\end{equation}
with arbitrary $c \in \CC$. With $c = \frac{\ii k }{4 \pi}$ the latter equality is obvious from 
(\ref{e:Da}). 
To prove it for arbitrary 
$c \in \CC$, note that the first minor in the left upper corner of the determinant in  (\ref{e:pa_betaNDa}) is equal to
$\det \Ga_{\wt \al, \wt Y} (k) = 0$.

Note that when $z \in \CC$ is fixed, $D_a (\cdot; z)$ is a polynomial in the variables $\beta_j$ 
and that $\pa_{\beta_j}^l D_a (\beta,z) = 0 $ for all $l \ge 2$. This implies the following lemma.

\begin{lem} \label{l:pa=0}
If $k \in \Si (a; Y)$ and $\pa_{\beta_j} D_a (b;k) =0$ for certain $1 \le j \le N$, then 
$k \in \Si (\al; Y)$ for all $\al $ obtained from $a$ by the change of the $j$-th coordinate $a_j$
to an arbitrary number in $\overline \CC$. \qed
\end{lem}

\section{Uniform logarithmic bound on resonances}
\label{s:UnifBound}

Let $N \ge 2$ and $\al \in \CC^N$. Then Lemma \ref{l:asDetGa} and its proof imply
the following 2-side  bound on all resonances $k \in \Si_\res (\al,Y)$: 
\begin{equation}
- \tfrac{N}{q_\nu  - q_{\nu-1} } \ln (|\re k|+1) - c_{2,2} \ \le \ \im k \ \le \ - \tfrac{2}{q_\nu } \ln (|\re k|+1) + c_{1,2},  \label{e:beLogDet}
\end{equation}
where $q_\nu$, $q_{\nu-1}$, $c_{2,2}$, $c_{2,1}$ are  positive constants the depending on $\al$ and $Y$ defined in the proof of Lemma \ref{l:asDetGa}.
The following theorem shows that the upper  bound can be modified in such a way that it becomes uniform 
with respect to $\al \in \overline{\RR}^N$ or $\al \in \overline{\CC}_-^N$.

\begin{thm}
\label{t:UniLogBound}
Let $N\ge 2$ and $\A \subset {\overline{\CC}_-} $. Then there exist $c_1 = c_1 (Y)>0$ such that
\begin{equation} \label{e:dr>}
\dr_{\min} (\fr; \F_\A) \ge \frac{2}{N \diam (Y)} \ln (|\fr|+1) - c_{1} 
\end{equation}
for all frequencies $\fr>0$ achievable  over $\F_{\A}$.
\end{thm}

\begin{proof}
\emph{Step 1.} As a function in $z$ and $\al\in \CC^N$, $\det \Ga_{\al,Y} (z)$ has the following representation 
\begin{equation} \label{e:det=SExpAl}
\det \Ga_{\al,Y} (z) = (-4\pi)^{-N} \wt D (\al,z) , \quad \wt D(\al,z) = \sum_{l=0}^\eta P_l (\al,z) e^{\ii z Q_l} ,
\end{equation}
which is unique if we assume that  the numbers $\eta = \eta (Y) \in \NN \cup \{0\}$, $Q_l = Q_l (Y) $,  and the nontrivial polynomials $P_l $ in $z$ and $\al_j$ (with coefficient depending on $Y$) are such that $0 = Q_0 < Q_1 < \dots < Q_\eta$.
In this case, one sees that $P_0 (\al,z) =  \prod_{j=1}^N  (\ii z - 4 \pi \al_j )$
 and 
\begin{gather} \label{e:QeQl}
 \deg P_l \le N-2  \text{ for all } 1 \le l \le \eta. 
\end{gather}
In the same way as in the proof of Lemma \ref{l:asDetGa}, the assumption $N\ge 2$ implies $\eta\ge 1$. 

\emph{Step 2.} Consider the case $\A = \A_0 :=\CC_- \cup \RR$ and $\al \in \A_0^N$. Then all $\al_j$ are finite and we can use (\ref{e:det=SExpAl}).
 Denote $ P_{\min} (\al,z) := \min_{j\neq j'} |(\ii z - 4 \pi \al_j )(\ii z - 4 \pi \al_{j'} )|$.
It is easy to see that there exists $C_5 = C_5 (Q_\eta)>0$ such that for any $c_1 \ge C_5$ in the region 
\[
\Om_3   := \{ \ \im z \le 0, \ \re z>0 \ , \ \im z > -\frac{2}{Q_\eta} \ln (\re z +1) + c_1 \}
\]
the inequalities 
\begin{equation} \label{e:|z-al|>}
|\ii z - 4\pi \al_j| \ge \tfrac{|z|+1}{4} + |\al_j| \ge \tfrac{1}{4}
\end{equation}
hold for all $\al_j \in \A_0 $.
Hence, for all 
$(z,\al) \in \Om_3  \times \A_0$, we have  
\begin{gather} \label{e:Pmin>}
\text{
$
P_0 (\al,z) \neq 0$, \qquad 
$P_{\min} (z,\al) > (|z|+1)^2/16$,} \quad \text{ and  } \\
\notag
|\wt D| \frac{P_{\min} }{|P_0 |} \ge P_{\min} - \sum_{l=1}^\eta \frac{P_{\min} |P_l| }{|P_0|} e^{- Q_\eta \im z } \ge \frac{(|z|+1)^2}{16} - 4^{N-2} (N!-1) e^{- Q_\eta \im z } .
 \end{gather}
The last inequality follows from (\ref{e:|z-al|>}),  (\ref{e:QeQl}), and 
the Leibniz formula for the determinant $\det \Ga_{\al,Y} $. Choosing 
$c_1$ large enough, one can ensure that $|\wt D| \frac{P_{\min} }{|P_0 |}$ (and so $\wt D$)
have no zeros in $\Om_3 \times \A_0^N$ .

\emph{Step 3.} From Step 2, the perturbation formula (\ref{e:k_asy}), and Lemma \ref{l:LocDet},  one sees
that $\det  \Ga_{\al,Y} (z)$ has no zeros in the larger set $\Om_3 \times \A^N$ with $\A= \overline{\CC}_-$. 
Now, the statement of theorem follows from the obvious estimate 
$Q_\eta \le N \diam (Y)$.
\end{proof}

\begin{rem} \label{r:BoundOverR}
When $\A = \overline{\RR} $, the estimate (\ref{e:dr>}) is valid for all frequencies $\fr$ that are achievable  over $\F_{\overline{\RR}}$.  Indeed, 
(\ref{e:Ga*}) and Example \ref{ex:1center} imply that
\begin{gather} \label{e:SiSym}
\text{ $\Si (\al,Y)$ is symmetric w.r.t. $\ii \RR$ for $ \al \in \overline{\RR}^N $ (including  multiplicities),}\\
 \text{$\dr_{\min} (\fr;\F_{\overline{\RR}})$ is an even function, and  
$\dr_{\min} (0;\F_{\overline{\RR}}) = 0$}.  \label{e:Dr_Min_aprR}
\end{gather}
\end{rem}

Note also that it follows from (\ref{e:SiSym}) that 
$\Si [\F_{\overline{\RR}}]$ and $\Si_\res [\F_{\overline{\RR}}]$ 
are symmetric w.r.t. $\ii \RR$, 
and that 
$\Si [\F_{\overline{\RR}}] = \Si_\res [\F_{\overline{\RR}}] \cup \ii [0,+\infty)$ .
To see the last equality, it suffices to notice that the m-accretivity statement of Proposition \ref{p:Res} implies
\begin{equation} \label{e:nospectrum}
\emptyset = \CC_{\mathrm{I}} \cap \Si [\F_{\overline{\CC}_-}]  = 
(\CC_{\mathrm{I}} \cup \CC_{\mathrm{II}})
 \cap \Si [\F_{\overline{\RR}}]  
\end{equation} 
and that 
$\Si [\F_{\overline{\RR}}]$ contains the set $\ii \RR$ produced by the case $N=1$ of Example \ref{ex:1center} (i).
Similarly, $\Si [\F_{\overline{\CC}_-}]$ contains the set 
$ \ii \CC_+ \cup \ii \RR$ of Example \ref{ex:1center} (ii) .
For the minimal decay function $\dr_{\min}$ over $\F_{\overline{\CC}_-}$ 
one has 
$\dr_{\min} (\fr;\F_{\overline{\CC}_-}) = 0$ for all $\fr \le 0$.

\section{Extremal resonances over $\A = \overline{\RR} $ and $\A= \overline{\CC}_-$}
\label{s:NesCond}

We study first the boundary $\Bd \Si [\F_{\overline{\CC}_-} ]$, and  then from this study 
obtain results on resonances of minimal decay over $\F_{\overline{\RR}} $.
The idea behind this is that there are more possible perturbations 
of the parameter tuple $a$ inside of $\overline{\CC}_-^N$, than
in the case $a \in \overline{\RR}^N$. So the restrictions on the possible perturbations of 
$k$ over $\F_{\overline{\CC}_-}$ are stronger. However, it occurs that for every  
$k \in \Bd \Si [\F_{\overline{\CC}_-} ] $, there exists $a \in \overline{\RR}^N$ 
that generates $k$ in the sense that $k \in \Si_\res (a,Y)$. As a result the resonances of minimal decay over  $\F_{\overline{\RR}}$ 
inherit the stronger necessary conditions of extremity derived for $\A =\overline{\CC}_-$.

In the simplest form, our main abstract result states that if 
$a \in  \RR^N$ generates the resonance $k$ of minimal decay over $\F_{\overline{\RR}}$, 
then there exists $\xi \in [-\pi, \pi)$
such that 
\begin{equation} \label{e:simpleNesCond}
\text{the ray $e^{\ii \xi} [0,+\infty) $ contains all first minors $\det \Ga^{[i]}_{a, Y} (k)$ (see Theorem \ref{t:pa_in_exi_R})}.
\end{equation} 
(Note that the ray $e^{\ii \xi} [0,+\infty)  \subset \CC$ includes its vertex at $0$.)

To formulate the result in the general form that includes the possibility of 
$a_j = \infty$ for some $j$ and the case $\A= \overline{\CC}_-$,
we take the convention of Section \ref{s:Da}, 
which assumes that the centers $y_j$ are enumerated 
in such a way that $a_j \neq \infty $ for $1 \le j \le n$ and $a_j =\infty$ for $n< j \le N$, and use 
the regularized determinant $D = D_a$ defined by (\ref{e:Da}) and depending on the modified parameters 
$\beta_j$.

\begin{thm} \label{t:pa_in_exi}
Assume that $a \in \overline{\CC}_-^N $ and $k \in \Si (a,Y)$ are such that 
$k \in \Bd \Si [\F_{\overline{\CC}_-}] $. Then the following statements hold:
\item[(i)] There exists $\xi \in [-\pi, \pi)$ such that
\begin{equation} \label{e:pa_in_exi}
\text{$\pa_{\beta_j} D (b;k)$ belongs to the ray $e^{\ii \xi} [0,+\infty)  $ for all $j=1,\dots, N$.}
\end{equation}
\item[(ii)] If $a_i \in \CC_-$ for some $i$, 
then $\pa_{\beta_i} D (b;k) = 0$ and 
\begin{equation} \label{e:pa=)a'}
\text{$k \in \Si (a',Y)$ for $a'$ defined by 
$a'_j = \left\{ \begin{array}{cc} c , & j=i \\ a_j, & j \neq i \end{array} \right.  $
with arbitrary $c \in \overline{\CC}_-$.}
\end{equation}
\item[(iii)] There exists $a' \in \overline{\RR}^N$ such that $k \in \Si (a',Y)$ 
and $n (a') = n_{\min} (k;\F_{\overline{\CC}_-})$. (Recall that $n_{\min}$ is defined by (\ref{e:nmin}).)
\end{thm}

The proof is given in Section  \ref{s:Proof_pa_in_exi}. Note that
Statement (iii) and Theorem \ref{t:SAclosed} imply
\begin{equation} \label{e:BdSsubsetBdS}
\Bd \Si [\F_{\overline{\CC}_-}] \subset \Bd \Si [\F_{\overline{\RR}}] \subset 
\Si [\F_{\overline{\RR}}] 
\subset \Si [\F_{\overline{\CC}_-}].
\end{equation}
On the other hand, the m-accretivity statement of Proposition \ref{p:Res} implies 
$\CC_{\mathrm{I}} \cap \Si [\F_{\overline{\CC}_-}] = \emptyset $. 
Combining this with (\ref{e:BdSsubsetBdS}), we obtain the next corollary.

\begin{cor}
\label{c:AC_AR}
Let $\fr \ge 0$. Then
the frequency $\fr$ is achievable  over $\F_{\overline{\CC}_-} $ exactly when 
it is achievable  over $\F_{\overline{\RR}}$. For such frequencies $\fr$, one has 
$
\dr_{\min} (\fr;\overline{\RR}) = \dr_{\min} (\fr;\overline{\CC}_-).
$ 
\qed
\end{cor}

With the use of Corollary \ref{c:AC_AR} and Theorem \ref{t:pa_in_exi}  
we will prove  in Section  \ref{s:Proof_pa_in_exi} the following necessary conditions 
over $\F_{\overline{\RR}}$.

\begin{thm} \label{t:pa_in_exi_R}
(i) If $k$ and $a \in \overline{\RR}^N$ are of minimal decay (over $\F_{\overline{\RR}}$) 
for the frequency $\re k $, then there exists $\xi \in [-\pi, \pi)$ such that
(\ref{e:pa_in_exi}) hold. 
\item[(ii)] If $a \in \overline{\RR}^N $ and $k \in \Si_\res (a,Y)$ are such that 
$k \in \Bd \Si_\res [\F_{\overline{\RR}}] $, then there exists $\xi \in [-\pi, \pi)$ such that
$\pa_{\beta_j} D (b;k)$ belongs to the line $e^{\ii \xi} \RR $ for all $j=1,\dots, N$.
\end{thm}

\subsection{Proofs of Theorems \ref{t:pa_in_exi} and \ref{t:pa_in_exi_R}}
\label{s:Proof_pa_in_exi}

Recall that by $\cone W$ we denote the nonnegative convex cone generated by a subset $W$ of a linear space. Let $S=\overline{\CC}_-^N$ or $S=\overline{\RR}^N$, and let $\F$ be defined by (\ref{e:F}).

Let $D (\beta;z) = D_a (\beta;z)$ and $b \in \CC^N$ be defined as in Section \ref{s:Da}.
The change from $\alpha$-coordinates to $\beta$-coordinates of Section \ref{s:Da} maps 
$S$ onto $S$.
Let $\wt S := S \cap \CC^N$ (this excludes all infinite points).
We would like to consider $\beta \in \wt S $
and the sets $\Si_D (\beta) $ of zeroes of $D (\beta, \cdot)$ generated by such $\beta$.
Obviously, 
\begin{equation} \label{e:SiDWinSi}
\text{the set of all such zeros, $\Si_D [\wt S ] := \bigcup_{\beta \in \wt S } \Si_D (\beta)$,
is a subset of $\Si [\F]$.}
\end{equation}

For $v \in \CC^N$, let us consider the directional derivatives 
$\frac{\pa D (b;z)}{\pa \beta} (v)  := \pa_\zeta D (b+v\zeta;z)$, where $\zeta \in \RR$.
Put $\frac{\pa D (b;k)}{\pa \beta} [\wt S - b] := 
\{ \frac{\pa D (b;z)}{\pa \beta} (v) \ : \ v \in \wt S - b \}$.
When $v$ belongs to the convex  set $\wt S - b $, the linear perturbations $b+v\zeta$
for $\zeta \in [0,1]$ remain in $\wt S$. Hence, \cite[Theorem 4.1 and Proposition 4.2]{Ka14} imply 
that if $\cone \frac{\pa D (b;k)}{\pa \beta} [\wt S - b] = \CC$, then 
$k$ is an interior point of $\Si_D [\wt S ]$. Taking (\ref{e:SiDWinSi}) into account
we obtain our main technical lemma.

\begin{lem} \label{l:intSiD}
If $k \in \Bd \Si [\F]$, then there exists a closed half-plane $e^{\ii \xi} (\CC_+ \cup \RR) $ that 
contains $\cone \frac{\pa D (b;k)}{\pa \beta} [\wt S - b]$. \qed
\end{lem}

Let us prove \emph{statement (ii) of Theorem \ref{t:pa_in_exi}.}  Assume that $k \in \Bd \Si [\F_{\overline{\CC}_-}]$ and 
$a_i \in \CC_-$. So $b_i = a_i \in \CC_-$.
Then every $v \in \CC^N$ such that $v_j = 0 $ for $j\neq i$ belongs to $\wt S - b$ for small enough $v_i \in \CC$ .  Note that 
$\frac{\pa D (b;k)}{\pa \beta} (v) =  v_i  \pa_{\beta_i} D (b;k)$. Lemma \ref{l:intSiD}
implies that there exists a half-plane  $e^{\ii \xi} (\CC_+ \cup \RR) $ that contains $\frac{\pa D (b;k)}{\pa \beta} (v)$ for all 
$v_i \in \CC$. This implies $\pa_{\beta_i} D (b;k) = 0$ and, in turn, due to 
Lemma \ref{l:pa=0}, implies (\ref{e:pa=)a'}). 

Now, \emph{statement (iii) of Theorem \ref{t:pa_in_exi}} follows from statement (ii).

Let us prove \emph{statement (i) of Theorem \ref{t:pa_in_exi}.}  For all $j$ such that 
$b_j \in \CC_-$, statement (ii) implies (\ref{e:pa_in_exi}) with arbitrary $\xi$.
For $j$ such that $b_j \in \RR$,  it is easy to see that $v = (c \de_{ji})_{i=1}^N$ is contained in 
$\wt S - b$ for arbitrary $c \in \overline{\CC}_- \setminus \{ \infty \}$. The corresponding derivatives 
$\frac{\pa D (b;k)}{\pa \beta} (v) = c \pa_{\beta_j} D (b;k) $ are contained in 
one complex half-plane only if  (\ref{e:pa_in_exi})  holds. 
Thus, (i) follows from Lemma \ref{l:intSiD}. This completes the proof of Theorem \ref{t:pa_in_exi}.

The above arguments and Lemma \ref{l:intSiD} allows one to obtain easily  \emph{Theorem \ref{t:pa_in_exi_R} (ii)}.
\emph{Theorem \ref{t:pa_in_exi_R} (i)} follows immediately from Theorem \ref{t:pa_in_exi} (iii),
 Corollary \ref{c:AC_AR}, and (\ref{e:SiSym}).

\section{Equidistant case}
\label{s:Tetrahedron}

In this section  an example where the configuration of location of the interactions is symmetric is studied in order to obtain explicit formulas for optimizers. Namely, we consider the case where $N=4$ and $y_j$ are vertices of a regular tetrahedron with edges of length $L$. 
We denote $n_{\min} (k) := n_{\min} (k;\F_{\overline{\RR}}) $ 
(see (\ref{e:nmin})).

\begin{thm} \label{t:EqLN=4} Let $N=4$. Assume that $|y_j - y_{j'}| = L$ for all $j \neq j'$. 
Then 
a frequency $\fr $ is achievable  over $\F_{\overline{\RR}}$ exactly when $\fr \neq \pm l \pi /L$, $l \in \NN$.
The minimal decay function $\dr(\fr) = \dr (\fr; \F_{\overline{\RR}})$ is given by the following explicit formulas:
\begin{gather*}
\textstyle \dr (0) = 0, \qquad  
\quad \dr (\fr)  = \frac{1}{L}\ln \frac{L\fr}{ \sin (L\fr)} \quad \text{ for } \ 
\fr \in \ \pm \bigcup_{l = 0}^{\infty} \ \left( \frac{ 2 l \pi}{L}, \frac{(2 l +1) \pi}{L} \right) , \\
\textstyle \text{ and } \quad \dr (\fr) = \frac{1}{L}\ln \frac{-L\fr}{3 \sin (L\fr)} \quad \text{ for } 
\  \fr \in \ \pm \bigcup_{l \in \NN} \ \left(  \frac{(2 l -1)\pi}{L}, \frac{2 l  \pi}{L}\right).
\end{gather*}
\end{thm}

\begin{figure}[h]
    \centering
\includegraphics[width=0.75\linewidth]{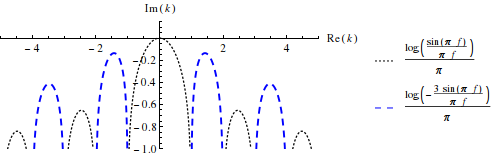}
    \caption{\footnotesize 
    Resonances 
    $k = \fr - \ii \dr (\fr)$ of minimal decay over $\F_{\overline{\RR}}$ in the equidistant case $L=\pi$, $N=4$;
-\,-\,- marks the case $n_{\min} (k) = 4$, \quad $\cdots$ marks the case $n_{\min} (k) = 2$ (see Lemmas \ref{l:nmin=4} and \ref{l:nmin=2}).
}
    \label{f:GraphDr}
\end{figure}

The rest of this subsection is devoted to the proof of Theorem \ref{t:EqLN=4}.

Let $n(a)=4$, $A_j := 4 \pi L a_j$ and $\vka:= L z$ (for the definition of $n(a)$ see (\ref{e:n})).
Then for  arbitrary $z$ and $a \in \RR^4$,
 \begin{gather}
 (-4 \pi L)^4 \det \Ga_{a,Y} (z) = 
\prod_{j=1}^4 (\ii \vka - A_j - e^{\ii \vka}) + 
e^{\ii \vka} \sum_{j=1}^4 \prod_{\substack{1 \le j' \le 4 \\ j' \neq j}}  (\ii \vka - A_{j'} - e^{\ii \vka}) \  \text{ and }
\label{e:4A_D} 
   \\
 \label{e:4A_D_paD} 
 (-4 \pi L)^4 \det \Ga_{a,Y} (z) = 
 (-4 \pi L)^3  ( \ii \vka - A_i - e^{\ii \vka}) \pa_{\al_i} \det \Ga_{a,Y} (z) 
 + e^{\ii \vka} \prod_{j' \neq i}  (\ii \vka - A_{j'} - e^{\ii \vka}) .
\end{gather}

Assume now that $z=k \in \Si_\res (a,Y)$. Then  $\det \Ga_{a,Y} (k) = 0$, and so, (\ref{e:4A_D})
implies 
\begin{equation} \label{e:4Adet=0}
0= \prod_{j=1}^4 (\ii \vka - A_j - e^{\ii \vka}) + 
e^{\ii \vka} \sum_{j=1}^4 \prod_{j' \neq j}  (\ii \vka - A_{j'} - e^{\ii \vka}) .
\end{equation}

\begin{lem} \label{l:nmin=4} 
Assume that $n(a) = n_{\min} (k) = 4$ and $k \in \Si_\res (a,Y)$ is of minimal decay for 
$\fr \in \RR$. Then 
$\fr \in \pm \bigcup_{l \in \NN} (  (2 l -1)\pi /L, 2 l  \pi / L)$, all $a_j$ are equal to each other, and 
\begin{equation} \label{e:4An4_ak}
a_1 = \dots = a_4 = \frac{1}{4 \pi L} \ln \frac{-L\fr}{3 \sin (L\fr)}  - \frac{\fr \cot (L \fr)}{4 \pi}, 
\text{\quad \quad}  k = \fr + \ii \frac{1}{L}\ln \frac{3 \sin (L\fr)}{-L\fr} .
\end{equation}
\end{lem}

\begin{proof}
Since $n(a) = 4$, we see that $a_j \in \RR$, $\beta_j = \alpha_j$, for $j=1,\dots, 4$,
 and $D_a (\alpha;z) =\det \Ga_{\al,Y} (z) $. 
By Example \ref{ex:1center}, $n_{\min} (k) =1 $ for all $k \in \ii \RR$.
So $n_{\min} (k) =4$ yields that $\fr = \re k \neq  0$. 
 Formula (\ref{e:4A_D_paD}) implies that for each 
$i$ either $\ii \vka - A_i - e^{\ii \vka} = 0$, or 
 \begin{gather} \label{e:paD_i_k=0}
 (-4 \pi L)^3 \pa_{\al_i} \det \Ga_{\al,Y} (k) = 
- \frac{e^{\ii \vka} \prod_{j' \neq i}  (\ii \vka - A_{j'} - e^{\ii \vka})}
 {  ( \ii \vka - A_i - e^{\ii \vka}) }.
\end{gather}

Let us show that in the case $n_{\min} (k) =4$, one has 
\begin{equation} \label{e:ik-A<>0}
\text{$\ii \vka - A_i - e^{\ii \vka} \neq 0$ 
for all $i$.}
\end{equation}
Assume that $\ii \vka - A_i - e^{\ii \vka} = 0$ holds for certain $i$. Then 
(\ref{e:4Adet=0}) yields that there exists 
$i' \neq i$ such that $\ii \vka - A_{i'} - e^{\ii \vka} = 0$. Hence, 
$A_{i'} = A_{i} $ and, for 
$a' $ defined by $a'_j = \left\{ 
\begin{array}{cc} a_j , & j=i,i' \\ \infty, & j \neq i,i' \end{array} 
\right.  $,
we have (taking into account the convention of Section \ref{s:ResAndOpt})
 \begin{gather} \label{e:nmin=2_A=A}
\det \Ga_{a',Y} (k) =  \det \Ga_{\wt a', \wt Y} (k)
= (-4 \pi L)^{-2}  (\ii \vka - A_i - e^{ \ii \vka})(\ii \vka - A_i + e^{ \ii \vka}) = 0 .
\end{gather}
This means 
$n_{\min} (k) \le 2$, a contradiction.

Thus, we see that (\ref{e:paD_i_k=0}) and $\pa_{\al_i} \det \Ga_{\al,Y} (z) \neq 0$ hold for all $i=1,\dots,4$.
To combine these conditions with Theorem \ref{t:pa_in_exi_R} (i), 
assume now that $k$ is of minimal decay for the frequency $\re k$. 
Then (\ref{e:pa_in_exi}) and (\ref{e:paD_i_k=0}) imply that for arbitrary $i \neq j$,
 \begin{gather} \label{e:fr_paj_pai}
\frac{\pa_{\al_i} \det \Ga_{\al,Y} (k)}{\pa_{\al_{j}} \det \Ga_{\al,Y} (k)} = 
\frac{  (\ii \vka - A_{j} - e^{\ii \vka})^2}
 {  ( \ii \vka - A_i - e^{\ii \vka})^2 } \in \RR_+ ,
\end{gather}

Let us show that $a_j = a_{j'} $ for all $j,j' = 1, \dots, 4$. Assume that the converse is true.
Then $A_{i} \neq A_{i'}$ for certain $i$ and $i'$.
However, (\ref{e:fr_paj_pai}) implies that $\frac{  \ii \vka - A_{j} - e^{\ii \vka} }
 {   \ii \vka - A_{j'} - e^{\ii \vka}} \in \RR \setminus \{0\} $ for all $j$ and $j'$.
So there exists $\wt \xi \in \RR$ such that  
$ \ii \vka - A_{j} - e^{\ii \vka} = c_j e^{\ii \wt \xi}$ with $c_j \in \RR \setminus \{0\}$ for all $j$.
Since $A_i , A_{i'} \in \RR$ and $0 \neq A_i - A_{i'} = (c_{i'} - c_i) e^{\ii \wt \xi} $, we see that 
$ e^{\ii \wt \xi} \in \RR$, and, in turn, 
$\ii \vka - e^{\ii \vka} \in \RR $. 
Combining this with 
(\ref{e:4Adet=0}) and (\ref{e:ik-A<>0}), one gets $e^{\ii \vka} \in \RR$, and in turn,
gets $\ii \vka \in \RR$ from $\ii \vka - e^{\ii \vka} \in \RR $. Finally, note that $\ii \vka \in \RR$ contradicts $\re k \neq 0$.


Summarizing, we have proved that if $k$ is of minimal decay and $n_{\min} (k) = 4$, then   
all $A_j  $ are equal to the same number, which we denote by $c$. Due to (\ref{e:ik-A<>0}), equality (\ref{e:4Adet=0}) turns  into  
$ c = \ii \vka + 3 e^{\ii \vka}  $. 
Since $c \in \RR$, taking $\re (\cdot)$ and $\im (\cdot)$ of the last equality 
we can derive an explicit relation 
between $\vka_1 := \re \vka \neq 0$, $\vka_2 :=  \im \vka \ge 0$, and $c$ 
using the arguments similar to that of the example in \cite[Section II.1.1]{AGHH12} (see also \cite{AH84,Sh85}).
Indeed, taking $\im (\cdot)$ one obtains 
$\vka_1 + 3 e^{-\vka_2} \sin \vka_1 = 0$ and, in turn, that 
$ \vka_1 \in \pm \cup_{l \in \NN} (  (2 l -1)\pi, 2 l  \pi)$
and $\vka_2 =  \ln \frac{3 \sin \vka_1}{-\vka_1} $. This gives the second part of (\ref{e:4An4_ak}).
The value of $c = 4 \pi L a_j$ is found by taking $\re (\cdot)  $.
\end{proof}


\begin{lem} \label{l:nmin=3} 
Assume that $k \in \Si_\res (a,Y)$ is of minimal decay for 
$\fr \in \RR$, $a \in \overline{\RR}^4$, and $n(a) = 3$. Then $n_{\min} (k) \le 2 $.
\end{lem}

\begin{proof} Assume that $n(a) = n_{\min} (k) =3$. 
Let us enumerate $y_j$ such that $a_j \in \RR$ and $\beta_j = \alpha_j$ for $j=1,2,3$.
Then $a_4= \infty$ and $\beta_4 = -1/\alpha_4$.
Since $D_a ((\beta_1, \beta_2, \beta_3, 0) ;z) =\det \Ga_{\wt \al, \wt Y} (z) $, we obtain 
analogously to
(\ref{e:4Adet=0}) and (\ref{e:4A_D_paD}) 
that for $i=1,2,3$,
\begin{gather} \label{e:4Adet=0_n=3} \textstyle
0= \prod_{j=1}^3 (\ii \vka - A_j - e^{\ii \vka}) + 
e^{\ii \vka} \sum_{j=1}^3 \prod_{\substack{1 \le j \le 3 \\ j' \neq j}} 
 (\ii \vka - A_{j'} - e^{\ii \vka})  \quad \text{ and } \\
 \textstyle
 0 = (-4 \pi L)^2 ( \ii \vka - A_i - e^{\ii \vka}) \pa_{\beta_i}  D_a  (k) 
 + e^{\ii \vka} \prod_{\substack{1 \le j' \le 3 \\ j' \neq i}}   (\ii \vka - A_{j'} - e^{\ii \vka}). \notag
\end{gather}

In the same way as in the proof of Lemma \ref{l:nmin=4}, one can show that 
$n(a) = n_{\min} (k) = 3$ implies 
$\ii \vka - A_j - e^{\ii \vka} \neq 0$  and  for $j=1,2,3$,
\begin{gather} \label{e:paD_i_k=0_n=3}
 (-4 \pi L)^2 \pa_{\beta_j} D_a (k) = 
 -\frac{e^{\ii \vka} \prod_{\substack{1 \le j' \le 3 \\ j' \neq j}}  (\ii \vka - A_{j'} - e^{\ii \vka})}
 {  ( \ii \vka - A_j - e^{\ii \vka}) } \neq 0
\end{gather}
For $j=4$, (\ref{e:pa_betaNDa})  with $c=(4 \pi L)^{-1} e^{\ii \vka}$ implies 
\begin{gather} 
 - (4 \pi L)^4 \pa_{\beta_4} D_a (k) = 
e^{\ii \vka} \prod_{j=1}^3  (\ii \vka - A_j - e^{\ii \vka})  \neq 0.
\label{e:4A_D_b4} 
\end{gather}

Since $k$ is of minimal decay, we obtain from Theorem \ref{t:pa_in_exi_R} (i) 
that for $j=1,2,3,$
 \begin{gather*} 
\text{$(4 \pi L)^2 \frac{\pa_{\beta_{4}} D_a (k) }{\pa_{\beta_j} D_a (k)}= 
(\ii \vka - A_{j} - e^{\ii \vka})^2 \in \RR_+$ and so $\ii \vka - A_{j} - e^{\ii \vka} \in \RR \setminus \{0\}$.}
\end{gather*}
Hence, $\ii \vka - e^{\ii \vka} \in \RR $ and, like in Lemma \ref{l:nmin=4},
one obtains from (\ref{e:4Adet=0_n=3}) that $e^{\ii \vka}$ and $\ii \vka$ are real.
The latter implies $k \in \ii \RR$ and, in turn, $n_{\min} (k) = 1$, a contradiction.
\end{proof}

\begin{lem} \label{l:nmin=2} (i) Assume that $k \in \Si_\res (a,Y)$ is of minimal decay for 
$\fr \in \RR$, $a \in \overline{\RR}^4$, and $n(a) = n_{\min} (k) = 2$. Then: 
\item[\quad (i.a)] $\fr \in \pm \bigcup_{l = 0}^{\infty} (  \frac{2 l \pi}{L}, \frac{(2 l +1) \pi}{L})$ and $k = \fr +  \frac{\ii}{L}\ln \frac{ \sin (L\fr)}{L\fr} $;
\item[\quad (i.b)] it is possible to enumerate $y_j$ so that 
$a_1= a_2=\frac{1}{4 \pi L} \ln \frac{L \fr}{\sin (L \fr)} - \frac{\fr}{4 \pi} \cot (L\fr)$ \quad and \quad $a_3=a_4= \infty$.

\noindent (ii) If $f$, $k$, and $a$ satisfy (i.a)-(i.b), then $k$ and $a$
are of minimal decay for $f$.
\end{lem}

\begin{proof}
\emph{(i)} As before, let us enumerate $y_j$ such that $a_j \in \RR$,  $j=1,2$.
Using the fact that $n_{\min} (k) =1$ for $k \in \ii \RR$, one shows in a way similar to the proof of Lemma \ref{l:nmin=3}, 
that $\ii \vka \not \in \RR$, $A_1 = A_2$, $\ii\vka - A_1 - e^{\ii \vka} \in \RR$, and that $\ii\vka - A_1 + e^{\ii \vka} \neq 0$.
Then (\ref{e:nmin=2_A=A}) implies that $\ii\vka - A_1 - e^{\ii \vka} = 0$.
Taking imaginary and real parts of $ A_1=\ii\vka  - e^{\ii \vka}$ in a way similar to the proof 
of Lemma \ref{l:nmin=4}, one gets 
(i.a-b).

\emph{(ii)} Suppose (i.a) and (i.b). It is obvious that $k \in \Si (a,Y)$ and so $f$ is an achievable  frequency. By Theorem \ref{t:SAclosed}, there exists resonance $k_0$ of minimal decay for $f$. The facts that $n_{\min} (k_0) \neq 1,3,4 $ follow, resp., from $f\neq 0$, Lemma \ref{l:nmin=3}, and the facts that $f$ is not in the frequency range of Lemma \ref{l:nmin=4}. Thus, $n_{\min} (k_0) = 2 $, and statement (i) of the lemma implies $k_0 =k$.
\end{proof}

Combining arguments of the proof of Lemma \ref{l:nmin=2} (ii) with Lemmas \ref{l:nmin=4}, \ref{l:nmin=3} and \ref{l:nmin=2} (i), it is easy to show that
$\RR \setminus \re \Si_\res [\F_{\overline{\RR}}] = \frac{\pi}{L} (\ZZ \setminus \{0\})$  and the fact that
\begin{equation} \label{e:nmin=2Inv}
\text{if $f$, $k$, and $a$ are as in Lemma \ref{l:nmin=4}, then $k$ and $a$ are of minimal decay for $f$.}
\end{equation} 
This completes the proof of Theorem \ref{t:EqLN=4}.

\section{Additional remarks and discussion}
\label{s:AddRem}

\textbf{Achievable frequencies}. Generically, if $N\ge 2$, the set  of achievable  frequencies $\re \Si_\res [\F_{\overline{\RR}}]$ takes the whole line $\RR$. More precisely, the example with two centers at the end of \cite[Section II.1.1]{AGHH12} easily implies the following statement.

\begin{prop} \label{p:AdmFr}
Let $N \ge 2$.
If $\fr_0 \in \RR \setminus \re \Si_\res [\F_{\overline{\RR}}]$, then 
for all $1 \le j,j' \le N$, $j \neq j'$, there exist 
nonzero integers $l_{j,j'} $ such that $\fr_0=\dfrac{ \pi l_{j,j'}}{|y_j - y_{j'}|}$. In particular, the set $\RR \setminus \re \Si_\res [\F_{\overline{\RR}}]$ either consists of isolated points, or is empty.  \qed
\end{prop} 
\vspace{0.5ex}

 \textbf{Minimization of the resonance width $\GaWidth$.} 
The interpretation of resonances $k$ from the point of view of the Schrödinger equation $\ii \pa_t u = H_\al u$ is usually done in another system of parameters. 
Namely, $E  =\re k^2$ is interpreted as \emph{the energy of the resonance} $k$ 
and $\GaWidth = 2 |\im k^2|$ is \emph{the width of the resonance} 
(see 
 e.g. \cite{RSIV78}). For nonnegative potentials with constraints on their $\L^p$-norms and compact supports, the problem of finding local and global minimizers of $\GaWidth$ was considered in \cite{HS86,S87}. 

The results of previous sections can be easily adapted to the problem of minimization of resonance width. The analogue of the problem of \cite{HS86} for point interactions  can be addressed in the following way. 

\begin{cor} \label{c:Width} Let $0 \le E_1 \le E_2 \le +\infty $ and $N\ge 2$. 
Then: 
\item[(i)] There exists $a \in \overline{\RR}^N$ and $k_0 \in \Si_\res (a,Y)$ such that $|\im k_0^2| 
= \inf_{\substack{k \in  \Si_\res [\F_{\overline{\RR}}] \\ E_1 \le \re k^2 \le E_2 }} |\im k^2| $. 
\item[(ii)] For any $a$ and $k_0$ satisfying (i), the necessary condition (i) of Theorem \ref{t:pa_in_exi} holds.
\end{cor}

\begin{proof}[Proof of statement (i)] It follows from Example \ref{ex:1center} and 
 the example in \cite[Section II.1.1]{AGHH12} that for any $E \ge 0$  
and any two-point set $Y'= \{ y'_1 , y'_2 \}$ there exist
a tuple $\al' \in \overline{\RR}^2$ and a resonance $k \in \Si_\res (\al',Y')$ such that $E=\re k^2$ (see also Section \ref{s:Tetrahedron}). So, when $N \ge 2$, the existence of minimizer follows in the case $E_2<+\infty$ from Theorem \ref{t:SAclosed}, 
and in the case $E_2 = +\infty$ from Theorem \ref{t:SAclosed} and the uniform bound (\ref{e:dr>}).
\end{proof}

The statement (ii) of Corollary \ref{c:Width} follows immediately from the following strengthened version of Theorem \ref{t:pa_in_exi_R} (i).

\begin{thm} \label{t:ConnComp} Denote by $\Om_+$ the path-connected components of the open set $\CC \setminus \Si_\res [\F_{\overline{\RR}}]$ that contain $\CC_+$. If $k \in \Bd \Om_+$ and $a \in \overline{\RR}^N$ are such that $k \in \Si_\res (a,Y)$, then there exists $\xi \in [-\pi, \pi)$ such that (\ref{e:pa_in_exi}) holds. 
\end{thm}

The proof of Theorem \ref{t:ConnComp} follows the same lines as that of Theorem \ref{t:pa_in_exi_R} (i).

\vspace{1ex}

\textbf{Symmetries and non-uniqueness of extremizers.} If there exists a unique operator $H$ that generates an extremal (in any sense) resonance, then $H$ preserves all the symmetries of  this optimization problem. 

This obvious principle can be illustrated by the tetrahedron equidistant case of Theorem \ref{t:EqLN=4} if we consider a  frequency $\fr \in \pm \bigcup_{l \in \NN} (  (2 l -1)\pi /L, 2 l  \pi / L)$ and  operators $H_a \in \F_{\overline{\RR}}$ of minimal decay for $f$. Indeed,  the tuple $a$ found in Lemma \ref{l:nmin=4} is the unique tuple of minimal decay for $f$.
The corresponding optimal operator $H_a$ possesses all the symmetries of the symmetry group $T_d$ of a regular tetrahedron.  The corresponding resonance of minimal decay is simple (i.e., of multiplicity 1).


\begin{ex} \label{ex:Tetrahedron} 
In the case $\fr \in \pm \bigcup_{l = 0}^{\infty} (  \frac{2 l \pi}{L}, \frac{(2 l +1) \pi}{L})$ (with $N=4$ and $|y_j - y_{j'}| = L$ for all $j \neq j'$),  the situation is different since an infinite family of generic $H_\al$ of minimal decay preserves only one of the symmetries. Let us consider in more details operators $H_a$ that generate the resonance $k$ of minimal decay over $\F_{\overline{\RR}}$ for such $\fr$
(this $k$ is calculated in Lemma \ref{l:nmin=2}).

It is easy to see that a 4-tuple $a \in \overline{\RR}^4$ is of minimal decay for $\fr$ if and only if 
\begin{equation} \label{e:DescrOptim_case_n=2}
\text{two of the parameters $a_1$, \dots, $a_4$ are equal  to
$\textstyle a_* := \frac{1}{4 \pi L} \ln \frac{L \fr}{\sin (L \fr)} - \frac{\fr}{4 \pi} \cot (L\fr)$.}
\end{equation}

Indeed,  in the case (\ref{e:DescrOptim_case_n=2}),
one can see that at least two of the numbers $A_j$  satisfy $\ii\vka - A_j - e^{\ii \vka} = 0$. So (\ref{e:4A_D}) and Lemma \ref{l:LocDet} imply $k \in \Si_\res (a,Y)$.
On the other hand, assume that $a$ does not satisfy (\ref{e:DescrOptim_case_n=2}). Then (\ref{e:4Adet=0}), (\ref{e:4Adet=0_n=3}), and (\ref{e:nmin=2_A=A}) imply $\ii\vka - A_j - e^{\ii \vka} \neq 0$ for all $j=1,\dots,4$.
Applying the arguments of Lemmas \ref{l:nmin=4} - \ref{l:nmin=2} (i), it is not difficult to see that $n(a) \neq 4,3,2,1$, a contradiction.
\end{ex}

We see  that, in the case of vertices of a regular tetrahedron and  $\F=\F_{\overline{\RR}}$,
\begin{itemize}
\item[(a)] each operator $H_a$ of minimal decay has at least one of the symmetries (of the symmetry group) of $\F$,
\item[(b)] there exists one $H_a$ of minimal decay that possesses all the symmetries of  $\F$.
\end{itemize}
The question to what extent the above observations (a) and (b) remain true for other feasible sets $\F$ and other resonance optimization problems \cite{CZ95,CO96,KS08} seems to be natural. We would like to note that related questions often appear in numerical and engineering studies \cite{KS08,LJ13,NKT08}.

\begin{rem} It worth to note that an example of a 1-D resonance optimization problem that possesses two different optimizers generating the same resonance of minimal decay  has been constructed recently in \cite{KLV17_Im}.  This example involves the equation of an inhomogeneous string and uses essentially the specific effects for its resonances on $\ii \RR_-$. 
\end{rem}
 
\textbf{Multiple resonances of minimal decay.} In many reasonable settings generic resonances are simple \cite{DZ16}. Resonances of minimal decay are very specific ones. Section \ref{s:Tetrahedron} shows that they can be multiple (i.e., of multiplicity $\ge 2$).

\begin{ex} \label{ex:mult}
In the settings of Example \ref{ex:Tetrahedron}, formula (\ref{e:4A_D}) implies that the resonance of minimal decay $k$ for $\fr \in \pm \bigcup_{l = 0}^{\infty} (  \frac{2 l \pi}{L}, \frac{(2 l +1) \pi}{L})$ is
\item[(i)] of multiplicity $2$ for $H_a$ if and only if exactly 
three of parameters $a_1$, \dots, $a_4$ are equal to $a_*$; 
\item[(ii)] of multiplicity $3$ for $H_a$ exactly when $a_1=\dots=a_4 = a_*$.
\end{ex}

It seems that the above effect with existence of multiple resonances of minimal decay is new. 
The explicitly computed 1-D resonances of minimal decay  in \cite{Ka14,KLV17_Im} are simple. However, it was noticed in numerical optimization experiments of \cite{KS08} that, in the 2-D case with upper and lower constraints on the index of refraction, the gradient ascent iterative procedure stopped when it encountered a multiple resonance because it was not able to determine which resonance branch to follow. In our opinion the Schrödinger operators with a finite number of point interactions is a good choice of a model for the study of the phenomena behind this numerical difficulty.

\begin{rem} As it was pointed out by the referee, the optimization technique of this paper can be applied to other types of non-selfadjoint spectral problems, e.g., to resonances of non-compact quantum graphs with finitely many edges \cite{DEL10,KS99,KS03,PW11}.
\end{rem}


\noindent Sergio Albeverio\\
Institute for Applied Mathematics, Rheinische Friedrich-Wilhelms Universität Bonn,
and Hausdorff Center for Mathematics, Endenicher Allee 60,
D-53115 Bonn, Germany\\
E-mail: albeverio@iam.uni-bonn.de

\vspace{2ex}

\noindent Illya M. Karabash\\
 Humboldt Research Fellow at Mathematical Institute, Rheinische Friedrich-Wilhelms Universität Bonn,
 Endenicher Allee 60, D-53115 Bonn, Germany;\\
and \\
Institute of Applied Mathematics and Mechanics of NAS of Ukraine,
Dobrovolskogo st. 1, Slovyans'k 84100, Ukraine\\
E-mail:  i.m.karabash@gmail.com


\begin{thebibliography}{99}



\bibitem{AFH79} S. Albeverio, J.E. Fenstad,  R. Høegh-Krohn, 
Singular perturbations and nonstandard analysis. 
Trans. Amer. Math. Soc. 252 (1979), 275--295.

\bibitem{AGH82} 
S. Albeverio, F. Gesztesy, R. Høegh-Krohn, 
The low energy expansion in nonrelativistic scattering theory. 
In Annales de l'IHP Physique théorique, Vol. 37, No. 1, 1982,  1--28. 


\bibitem{AGHH12} S. Albeverio, F. Gesztesy, R. Høegh-Krohn, H. Holden,   
Solvable models in quantum mechanics. 2nd edition, with an appendix by P. Exner. AMS Chelsea Publishing, Providence, RI, 2005.

\bibitem{AGHS83} S. Albeverio, F. Gesztesy, R. Høegh-Krohn, L. Streit,  Charged particles with short range interactions. In Annales de l'IHP Physique théorique, Vol. 38, No. 3, 1983,  263--293.

\bibitem{AH84} 
S. Albeverio, R. Høegh-Krohn, Perturbation of resonances in quantum mechanics. J. Math. Anal. Appl. 101 (1984),  491--513.

\bibitem{AK00} S. Albeverio, P. Kurasov. Singular perturbations of differential operators: solvable Schrödinger-type operators. 
Cambridge University Press, 2000.

\bibitem{BG12} C.A. Berenstein, R. Gay R, Complex analysis and special topics in harmonic analysis. Springer Science \& Business Media, 2012.

\bibitem{BV04} S. Boyd, L. Vandenberghe, Convex optimization. Cambridge university press, Cambridge, 2004.



\bibitem{CO96} S.J. Cox, M.L. Overton, Perturbing the critically damped wave equation,
SIAM J. Appl. Math.
56 
(1996), 1353--1362.

\bibitem{CZ95}
S. Cox, E. Zuazua, The rate at which energy decays in a string
damped at one end, 
Indiana Univ. Math. J. 
 44
(1995), 
545--573.

\bibitem{DEL10} E.B. Davies, P. Exner, J. Lipovsk\'{y},   Non-Weyl asymptotics for quantum 
graphs with general coupling conditions, Journal of Physics A  43(47) (2010), 474013.

\bibitem{DMTSH14} U.P. Dharanipathy, M. Minkov, M. Tonin, V. Savona, R. Houdré, High-Q silicon photonic crystal cavity for enhanced optical nonlinearities. Applied Physics Letters 105 
(2014), 101101.

\bibitem{DZ16} S. Dyatlov, M. Zworski, Mathematical theory of scattering resonances. Book in progress, http://math. mit. edu/dyatlov/res/

\bibitem{E05} P. Exner,  An isoperimetric problem for point interactions. J. Phys. A 38
(2005), 4795.

\bibitem{E12} P. Exner, Open quantum systems and Feynman integrals. Springer Science \& Business Media, Berlin, 2012.

\bibitem{EFH07} P. Exner, M. Fraas, E.M. Harrell,  On the critical exponent in an isoperimetric inequality for chords. Physics Letters A 368 
(2007), 1--6.

\bibitem{EHL06} P. Exner, E.M. Harrell, M. Loss, Inequalities for means of chords, with application to isoperimetric problems. Lett. Math. Phys. 75
 (2006), 225--233.

\bibitem{EL16} P. Exner, V. Lotoreichik, A spectral isoperimetric inequality for cones.   Lett. Math. Phys. 107 (2017), 
717--732.

\bibitem{F97} R. Froese, Asymptotic distribution of resonances in one dimension. J. Differential Equations 137
(1997), 251--272.

\bibitem{G84} F. Gesztesy, Perturbation theory for resonances in terms of Fredholm determinants. In Resonances--Models and Phenomena,
 edited by S. Albeverio, L. S. Ferreira and L. Streit. Springer, Berlin--Heidelberg, 1984, 78--104

\bibitem{GCh13} Yiqi Gu, Xiaoliang Cheng, A Numerical Approach for Defect Modes Localization in an Inhomogeneous Medium.
SIAM J. Appl. Math.  
73 
(2013), 2188--2202.


\bibitem{H82} E.M. Harrell, General lower bounds for resonances in one dimension.
Comm. Math. Phys.
86
(1982), 221--225.

\bibitem{HS86} E.M. Harrell, R. Svirsky,  Potentials producing maximally sharp resonances,  
Trans. Amer. Math. Soc. 
293 
 (1986), 723--736.

\bibitem{HBKW08} P. Heider, D. Berebichez, R.V. Kohn, and M.I. Weinstein,
Optimization of scattering resonances, 
Struct. Multidisc. Optim.
 36
 (2008), 443--456.

\bibitem{KS08}
C.-Y. Kao, F. Santosa, Maximization of the quality factor of an optical resonator, 
 Wave Motion
  45
  (2008), 412--427.


\bibitem{Ka13} I.M. Karabash,
Optimization of quasi-normal eigenvalues for 1-D wave equations in inhomogeneous media;
description of optimal structures,  
Asymptotic Analysis
81 
(2013),
273--295.



\bibitem{Ka14} I.M. Karabash, Pareto optimal structures producing resonances of minimal decay under $L^1$-type constraints,
J. Differential Equations
257 
(2014), 
374--414.

\bibitem{KLV17} I.M. Karabash, O.M. Logachova, I.V. Verbytskyi, Nonlinear bang-bang eigenproblems and optimization of resonances in layered cavities.  Integr. Equ. Oper. Theory 88(1) (2017), 15--44.

\bibitem{KLV17_Im} I.M. Karabash, O.M. Logachova, I.V. Verbytskyi, Overdamped modes and optimization of resonances in layered cavities, to appear in Methods of Functional Analysis and Topology.
 

\bibitem{Kato13} T. Kato, Perturbation theory for linear operators. Springer Science \& Business Media, Berlin, 2013.

\bibitem{KS99} V. Kostrykin, R.  Schrader, Kirchhoff's rule for quantum wires. Journal of Physics A 32(4), (1999), 595.

\bibitem{KS03} T. Kottos, U. Smilansky,  Quantum graphs: a simple model for chaotic scattering. Journal of Physics A 36(12),  (2003), 3501--3524.



\bibitem{LP71} P.D. Lax, R.S. Phillips, A logarithmic bound on the location of the poles of the scattering matrix, 
Arch. Ration. Mech. Anal. 
40 
(1971), 268--280.


\bibitem{LJ13} X. Liang,  S.G.  Johnson, Formulation for scalable optimization of microcavities via 
the frequency-averaged local density of states, 
Optics express 
21 
(2013), 30812--30841.

\bibitem{L31} R.E. Langer,  On the zeros of exponential sums and integrals. Bulletin of the American Mathematical Society 37(4), (1931),  213--239.

\bibitem{L16} V. Lotoreichik,  Spectral isoperimetric inequalities for $\delta $-interactions on open arcs and for the Robin Laplacian on planes with slits. (2016), arXiv preprint arXiv:1609.07598.

\bibitem{MPBKLR13} B. Maes, J. Petr{\'{a}}{\v{c}}ek, S. Burger, P. Kwiecien, J. Luksch, I.  Richter,   Simulations of high-Q optical nanocavities with a gradual 1D bandgap, 
Optics express 
 21
 (2013), 6794--6806.
 
\bibitem{MSV13} G. Mora, J.M. Sepulcre, T. Vidal,  On the existence of exponential polynomials with prefixed gaps. Bulletin of the London Mathematical Society 45(6), (2013), 1148--1162.

\bibitem{NKT08} M. Notomi, E. Kuramochi, H. Taniyama, Ultrahigh-Q nanocavity with 1D photonic gap, 
 Optics Express
 16(15)
 (2008), 11095--11102.

\bibitem{KO14_Th} Kai Ogasawara, The analysis of resonant states 
in a one-dimensional  potential of quantum dots, Honours Physics Bachelor thesis,
University of British Columbia, 2014.

\bibitem{OW13} B. Osting,  M.I. Weinstein, Long-lived scattering resonances and Bragg structures. SIAM J. Appl. Math. 73
 (2013), 827--852.
 
\bibitem{PW11} V. Pivovarchik, H. Woracek,  Eigenvalue asymptotics for a star-graph damped vibrations problem. Asymptotic Analysis 73(3), (2011), 169--185.

\bibitem{RSIV78} M. Reed, B. Simon, Analysis of Operators, Vol. IV of Methods of Modern Mathematical Physics.  Academic Press, New York, 1978. 

\bibitem{Sh85} A.A. Shushkov,
 Structure of resonances for symmetric scatterers. Theoretical and Mathematical Physics 64
 (1985), 944--949.

\bibitem{S87} R. Svirsky, Maximally resonant potentials subject to p-norm constraints. Pacific J. Math. 129
(1987), 357--374.



\bibitem{V72} B.R. Vainberg,  Eigenfunctions of an operator that correspond to the poles of the analytic continuation of the resolvent across the continuous spectrum. (Russian) Mat. Sb. (N.S.) 87(129) (1972), 293--308. English transl.: in Mathematics of the USSR-Sbornik, 16(2)(1972).

\bibitem{Z89} M. Zworski, Sharp polynomial bounds on the number of scattering poles. Duke Math. J 59
 (1989),  311--323.



\end{thebibliography}
\end{document}